\documentclass{amsart}
\usepackage[T1]{fontenc}
\usepackage[english]{babel}
\usepackage{amsmath,amsthm,amssymb, latexsym,geometry,stackrel,enumerate}
\usepackage[all]{xy}
\usepackage[dvips]{graphicx}
\DeclareGraphicsExtensions{.bmp,.png,.pdf,.jpgg}
\usepackage[all]{xy}
\usepackage{verbatim}
\usepackage[mathcal]{euscript}
\usepackage{epsfig}
\usepackage{multicol}
\usepackage{color}
\usepackage{verbatim}
\usepackage{graphicx,float}
\usepackage{pstricks-add}

\swapnumbers
\theoremstyle{plain}% default
\newtheorem{teo}{Theorem}[section]
\newtheorem{lema}[teo]{Lemma}
\newtheorem{prop}[teo]{Proposition}
\newtheorem{coro}[teo]{Corollary}

\theoremstyle{definition}
\newtheorem{defi}[teo]{Definition}

\newtheorem{ejem}[teo]{Example}

\newtheorem{nota}[teo]{Notation}

\newcommand{\ts}[1]{\normalfont{\textsf{#1}}}
\newcommand{\K}{\ts k}

\renewcommand{\a}{\alpha}
\renewcommand{\b}{\beta}

\newcommand{\X}{\mathbb X}

\newcommand{\za}{\alpha}

\newcommand{\zd}{\delta}
\newcommand{\ze}{\epsilon}

\newcommand{\zl}{\lambda}
\newcommand{\zs}{\sigma}
\newcommand{\ind}{\textup{ind}\,}
\newcommand{\mesh}{\mathcal{M}}

\newcommand{\ZZ}{\mathbb{Z}}

\newcommand{\edges}{E'} % NOTATION FOR THE SET OF TAGGED EDGES

\newcommand{\pos}{\textup{pos}\,}
\newcommand{\level}{\textup{level}\,}

\input{xy}
\xyoption{poly}
\xyoption{2cell}
\xyoption{all}

\title[ Repetitive cluster categories of type $D_n$ ]{Repetitive cluster  categories of type $D_n$}

\author[V. Gubitosi]{Viviana Gubitosi}
\address{Instituto de Matem\'{a}tica y Estad\'{\i}stica Rafael Laguardia, Facultad de Ingenier\'{\i}a - UdelaR, Montevideo, Uruguay, 11200 }

\email{gubitosi@fing.edu.uy}

\keywords{Repetitive cluster categories; cluster categories; geometric model }

\begin{document}
\maketitle

\begin{abstract}
 In this paper, we show that the repetitive  cluster category of type $D_n$,
defined as the orbit category $\mathcal{D}^b(\mathrm{mod}\K D_n)/(\tau^{-1}[1])^p$, is
equivalent to a category defined on a subset of tagged edges in a regular punctured polygon. This
generalizes the construction of Schiffler, \cite{S}, which we recover when $p=1$.
\end{abstract}

\section*{Introduction}

Cluster categories were introduced in \cite{BMRRT} and, independently,
in \cite{CCS} for type $A_n$, in order to model the combinatorics of the
cluster algebras of Fomin and Zelevinski \cite{FZ1,FZ2}.

Roughly speaking, given an hereditary finite dimensional algebra $\mathcal{H}$  over an  algebraically closed field, the
cluster category $\mathcal{C}_{\mathcal{H}}$ is obtained from the derived category $\mathcal{D}^b(\mathcal{H})$ by identifying the shift functor $[1]$ with the Auslander–Reiten translation $\tau $.

For any positive integer $p$, the repetitive cluster category  $\mathcal{C}_{F^p}(\mathcal{H})$
was introduced  by Zhu, in \cite{Zhu},  for any hereditary abelian category $\mathcal{H}$ with
tilting objects. The Repetitive cluster category is defined as the orbit category of the bounded derived category
$\mathcal{D}^b(\mathcal{H})$  under the action of the cyclic group generated by the
auto-equivalence $F^p=(\tau^{-1}[1])^p$, where $\tau$ is
the Auslander-Reiten translation and $[1]$ is the shift functor. Repetitive cluster categories are triangulated \cite{Keller} and fractionally Calabi-Yau of dimension $\frac{2p}{p}$ \cite{Keller1} by Keller.

% i.e;
%$(\tau[1])^p\cong[2p]$ as triangle functors, and the fraction cannot be
%simplified.

The cluster tilting objects in these repetitive cluster categories are shown to correspond
one-to-one to those in the classical cluster categories. The endomorphism algebras of cluster tilting objects in
the repetitive cluster category are the coverings of the cluster tilted algebras \cite{Zhu}.

When $\mathcal{H}$ is an hereditary algebra of type $D_n$ we will denote its repetitive cluster category as ${\mathcal C}_{n,p} $. For $p=1$ we recover the usual cluster category of type $D_n$, which we denote simply by
$\mathcal{C}_n$.

Many authors have been extensively studied the association of geometric models to algebraic categories
in several contexts, see for instance \cite{BM2,BM,BuTo,CCS,Lisa1,LambertiRep,S,Tork}.
In particular repetitive cluster categories of type $A_n$ have been studied by Lamberti in \cite{LambertiRep}. She showed an equivalence  of categories between the repetitive cluster category of type $A_n$ and a category of diagonals in a regular $p(n+2)$-gon. The model proposed also leads to a geometric interpretation of cluster tilting objects in the repetitive cluster category for type $A_n$.

In this paper, we give a geometric realization of the repetitive cluster
categories of type $D_n$ in the spirit of \cite{CCS} and \cite{LambertiRep}.
We addapt the geometric description given by Schiffler in \cite{S}.

Our main result is the equivalence of the category of tagged edges and the repetitive
cluster category ${\mathcal C}_{n,p}$, see Corollary \ref{main result}.

 The article is organized as follows. After a preliminary section,
in which we fix the notations and recall some concepts needed
later, section 2 is devoted to recall  the definition of the repetitive cluster category
$\mathcal{C}_{n,p}$  an some basic properties. In section 3, we study the relation between
repetitive cluster categories  and  cluster categories. Section 4 is dedicated to the definition
of  the category $\mathcal{C}(\textsf{P}_{np})$  of tagged edges. We show the equivalence of this
category and the repetitive  cluster category in section 5. Finally, in section 6, we give a
geometric interpretation of cluster tilting objects in $\mathcal{C}_{n,p}$.

%;  we show that $\mathcal{C}_{n,p}$
%is triangulated equivalent to a quotient of the cluster category $\mathcal{C}_{t}$.

\section{ Definitions and Preliminaries} \label{notation}

%In order to explain this, let us give the precise
%definition of a translation quiver.

\subsection{Notation}

Let $\K$ be an algebraically closed field. A quiver $Q$ is the data of two sets, $Q_0$ (the vertices)
and $Q_1$ (the arrows) and two maps $s,t: Q_1\rightarrow  Q_0$ that assign to each arrow $\alpha$ its
source $s(\alpha)$ and its target $t(\alpha)$.
The {\em path algebra} of $Q$ over $\K$ will be denoted by $\K Q$.
By Gabriel \cite{G1},  $\K Q$ is of finite representation type (i.e there is only a finite  number of isoclasses of indecomposable modules)  if and only if
$Q$ is a Dynkin quiver, that is, the underlying graph of $Q$ is a
Dynkin diagram of type $A_n,D_n$ or $E_n$.\\

Given a finite dimensional algebra $A$ denote by $\mathcal{D}^b(\textup{mod}\, A)$ the
bounded derived category of finite dimensional right $A$-modules.
Its objects are bounded complexes of finite dimensional right $A$-modules,
and morphisms are obtained from morphisms of complexes by localizing with
respect to quasi-isomorphisms (see \cite{H}). The category $\mathcal{D}^b(\textup{mod}\, A)$ is triangulated, with
translation functor induced by the shift of complexes.
Throughout the paper we denote by $\mathcal{D}$ the bounded derived  category  $\mathcal{D}^b(\mathrm{mod}\K D_{n})$.

\subsection{Serre duality and Calabi-Yau categories}

Let $\K$ be a field  and let $\mathcal{K}$ be a $\K$-linear triangulated category which
is $\mathrm{Hom}$-finite, i.e.\ for any two objects in $\mathcal{T}$ the space of
morphisms is a finite dimensional vector space.

Remember from \cite{K2} that  a $\K$-triangulated category $\mathcal{K}$ has a {\em Serre
functor} if it is equipped with an auto-equivalence
$\nu:\mathcal{K}\rightarrow\mathcal{K}$ together with bifunctorial isomorphisms $$
\rm D\mathrm{Hom}_{\mathcal{K}}(X,Y)\cong\mathrm{Hom}_{\mathcal{K}}(Y,\nu X), $$ for each
$X,Y\in\mathcal{K}$, where  $\rm D$ indicates the vector space duality $\mathrm{Hom}_{\K}(-,\K)$.

We will say that $\mathcal{K}$ has Serre duality if $\mathcal{K}$ admits a
Serre functor. If $\mathcal{D}$ denotes the category $\mathcal{D}^b(\mathrm{mod}\K D_{n})$ and we consider the case $\mathcal{K}=\mathcal{D}$  a Serre functor exists
(\cite{K2} p. 24), it is unique up to isomorphism and $\nu\stackrel
{\sim}{\rightarrow}\tau[1]$, where $\tau$ is the Auslander-Reiten translate and $[1]$ is
the shift functor of $\mathcal{D}$.

For $n,m>0$,  a category $\mathcal{K}$ with Serre functor $\nu$ is said to be  {\em
fractionally Calabi-Yau  of dimension $\frac{m}{n}$} or
$\frac{m}{n}$-Calabi-Yau if there is an isomorphism of triangle functors:
$$\nu^n\cong[m]$$
where $[m]$ indicates the composition of the shift functor with itself $m$ times.

\subsection{Translation quivers}\label{mesh cat and mesh relations}

Following \cite{Riedt}, we define a {\em stable translation
 quiver}  to be a  pair $(\Gamma,\tau)$
where $\Gamma$ is a locally finite quiver and $\tau:\Gamma'_0\to
\Gamma_0$ is an injective map defined on a subset $\Gamma'_0$ of
the vertices of $\Gamma$ such that for any $x\in\Gamma_0$, $y\in\Gamma'_0$,
the number of arrows from $x$ to $y$ is the same as the number of
arrows from $\tau(y)$ to $x$.

$(\Gamma,\tau)$ is called a \emph{stable translation quiver} if $\Gamma'_0=\Gamma_0$ and $\tau$ is
bijective.

A stable translation quiver is said to be \emph{connected} if it is not a
disjoint union of two non-empty stable subquivers.

Given a stable translation quiver $(\Gamma,\tau)$, a polarization of $\Gamma$ is
 a bijection $\sigma:\Gamma_1\to\Gamma_1$ such that
$\sigma(\za):\tau\,x\to y$  for every arrow $\za:y\to x\in \Gamma_1$.
If $\Gamma $ has no  multiple arrows, then there is a unique polarization.

We remark that in all examples of stable translation quivers appearing
in this article, the number of arrows between two vertices is always at most $1$.\\

The {\em path category } of $(\Gamma,\tau)$ is the category whose  objects are
the vertices of $\Gamma$, and given $x,y\in \Gamma_0$, the $\K$-space of
morphisms from $x$ to $y$ is given by the $\K$-vector space with basis
the set of all paths from $x $ to $y$. The composition of morphisms is
induced from the usual composition of  paths.

The {\em mesh ideal} in the path category of $\Gamma$ is the ideal
generated by the {\em mesh relations}
\begin{equation}\nonumber
m_x =\sum_{\za:y\to x} \sigma(\za) \za.
\end{equation}

The {\em  mesh category } $\mesh (\Gamma,\tau)$ of $(\Gamma,\tau)$ is the
quotient of the path
category of $(\Gamma,\tau)$ by the mesh ideal.\\

Given a quiver $Q$ one can construct a stable translation quiver
$\mathbb{Z}Q$ as follows: $(\mathbb{Z}Q)_0=\mathbb{Z}\times Q_0$ and
the number of arrows in $\mathbb{Z}Q$ from $(i,x)$ to $ (j,y)$ equals the number of
arrows  in $Q$ from $x$ to $y$ if $i=j$, and equals the number of arrows in $Q$
from
$y$ to $x$ if $j=i+1$, and there are no arrows otherwise. The
translation $\tau$ is defined by $\tau((i,x))=(i-1,x)$.\\

Important examples of translation quivers are the Auslander-Reiten
quivers of the derived categories of hereditary algebras of finite
representation type. We shall need the following proposition.

\begin{prop} \cite[I.5]{H}\label{prop 1}
Let $Q$ be a Dynkin quiver. Then
\begin{enumerate}
%\item for any quiver $Q'$ of the same Dynkin type as $Q$, the derived categories $\mathcal{D}^b kQ$ and $
 %\mathcal{D}^b \K Q' $ are  equivalent.
\item the Auslander-Reiten quiver of $\mathcal{D}^b(\rm mod  \K Q)$
  is $\mathbb{Z}Q^{op}$.
\item the category $\ind\mathcal{D}^b (\rm mod \K Q)$ is equivalent
  to the mesh category of  $\mathbb{Z}Q^{op}$.
\end{enumerate}
\end{prop}

\subsubsection{The stable translation quiver of $\mathcal{D}$}

Let $Q$ be a quiver of underlying Dynkin type $D_n$. We denote the vertices of $Q$ with $0,\overline{0},1,\cdots, n-2$ and the arrows are $i-1\rightarrow i$ ($i=1, \cdots , n-2$) together with $\overline{0}\rightarrow 1$;  see Figure~\ref{quiverdn}.

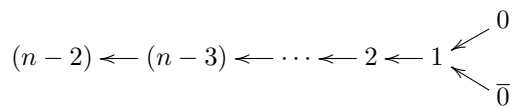
\begin{figure}[H]
\begin{center}
\[\xymatrix@R0pt@C14pt{
&&&&&& 0 \ar[dl]\\
& (n-2)
&(n-3)\ar[l]
&\cdots \ar[l]
&2\ar[l]
&1\ar[l]\\
&&&&&& \overline{0}\ar[lu].
}
\]
\caption{The quiver of type  $D_n$}
\label{quiverdn}
\end{center}
\end{figure}

By proposition \ref{prop 1}, the
Auslander-Reiten quiver of $\mathcal{D}$ is the stable translation
quiver $\ZZ Q^{op}$.  The labels of $Q_0$ induce labels
on the vertices of $\ZZ Q^{op}$ as usual:

\[ (\ZZ Q^{op})_0=\{(i,j)\mid i\in \ZZ, \ j\in Q_0\}=\ZZ\times\{0,\overline{0},1,\ldots,n-2\}.\]

%The number of arrows in $\mathbb{Z}Q$ from $(i,x)$ to $ (j,y)$ equals the number of
%arrows  in $Q$ from $x$ to $y$ if $i=j$, and equals the number of arrows in $Q$
%from
%$y$ to $x$ if $j=i+1$, and there are no arrows otherwise. The
%translation $\tau$ is defined by $\tau((i,x))=(i-1,x)$.

%otra manera:
%The arrows are given by
%$(i,j)\rightarrow (i,j-1)$ and
%$(i,j-1)\rightarrow (i+1,j)$ for $1\leq j\leq n-2$,
%and $(i,1)\rightarrow (i,\overline{0})$ and
%$(i,\overline{0})\rightarrow (i+1,1)$,
%where $i\in\mathbb{Z}$ is arbitrary.

We can identify the indecomposable objects of $\mathcal{D}$ with the
vertices of $\ZZ Q$ as follows. Let $P_1$ be the indecomposable projective module
corresponding to the vertex $1\in Q_0$. Then, by defining the position
of $P_1$ to be $(0,1)$, we have a bijection

\[ \pos : \ind \mathcal{D} \to \ZZ\times\{0,\overline{0},1,\ldots,n-2\}. \]

In other terms, for $M\in\ind \mathcal{D}$, we have $\pos (M)=(i,j)$ if and
only if $M=\tau^{-i}\,P_j$, where $P_j$ is the indecomposable
projective $A$-module at vertex $j$.
The "integer" $j\in\{\overline{0},0,\ldots,n-2\}$ is called the {\em level} of $M$ and
will be denoted by $\level(M)$.

If $\pos(M)=(i,j)$ with $j\in \{0,\overline{0}\}$ then
let $M^-$ be the indecomposable object such that $\pos(M^-)=(i,j')$,
where $j'$ is the unique element in $\{0,\overline{0}\}\setminus\{j\}$.

The structure of the derived
category $\mathcal{D}$ is well known.
In particular, we have the following result.

\begin{lema}\cite{S} Let  $M\in\ind\mathcal{D}$.
\begin{enumerate}
\item If $n$ is even, then $M [1]=\tau_{\mathcal{D}}^{-n+1}\,M$.
\item If $n$ is odd, then
\[M[1]=\left\{\begin{array}{ll}
\tau_{\mathcal{D}}^{-n+1}\,M &\textup{if  $\level(M)\notin\{0,\overline{0}\}$}\\
\\
\tau_{\mathcal{D}}^{-n+1}\,M^- &\textup{if $\level(M)\in\{0,\overline{0}\}$}
\end{array}\right. \]
\end{enumerate}
\end{lema}

\section{Repetitive cluster categories of type $D_n$ }

For $p>0$, the {\em repetitive cluster category} of type $D_n$ is defined as the orbit category
of the bounded derived category  $\mathcal{D}$ of $\mathrm{mod} \K D_{n}$ under the
action of the cyclic group generated by the auto-equivalence
$(\tau^{-1}[1])^p=\tau^{-p}[p]$:

$$\mathcal{C}_{n,p}=\mathcal{D}/<\tau^{-p}[p]>$$

where $\tau$ is the AR-translation in $\mathcal{D}$ and $[1]$ is the shift functor.

Thus the objects $\tilde M$ of $\mathcal{C}_{n,p}$ are the $(\tau^{-1}[1])^p$-orbits
$\tilde M :=((\tau^{-p}[p])^i\, M)_{i\in\ZZ}$ of
objects $M\in \mathcal{D}$. For simplicity later on we will omit the tilde.

For two objects $\tilde M, \tilde N $ in $\mathcal{C}_{n,p}$ the class of morphism is given by

$$\mathrm{Hom}_{\mathcal{C}_{n,p}}(\tilde M,\tilde N)= \bigoplus_{i\in \mathbb{Z}}\mathrm{Hom}_{\mathcal{D}}(M, (\tau^{-p}[p])^i N) $$

Repetitive cluster categories were defined by
 Zhu \cite{Zhu} for any hereditary abelian category with tilting objects. In case $p=1$, this definition specializes to that of the usual cluster category that we denote simply $\mathcal{C}_n$.

Let $\mathcal{F}=\mathcal{F}_1$ be the fundamental domain for the
functor $F:=\tau^{-1}[1]$ in $\mathcal{D}$. It is given by the isoclasses of indecomposables
objects in $\mathrm{mod} \K D_n$  together with the $[1]$-shift of the projective
indecomposable modules.

Let $\mathrm{ind}(\mathcal{C}_n)$ denote the subcategory of isomorphism classes of indecomposable objects of $\mathcal{C}_n$. Then, after \cite[Proposition 1.6]{BMRRT}, $\mathrm{ind}(\mathcal{C}_n)$ can be identified with the objects in $\mathcal{F}$.
Let $F^k$ be the composition of $F$ with itself $k$-times.  Denote by $\mathcal{F}_k$ the image under  $F^k$
of the fundamental domain $\mathcal{F}$. Then we can draw the fundamental domain for the functor
$\tau^{-p}[p]$ as in Figure \ref{funddom}.

\begin{figure}[H]
 \unitlength=1cm
   \begin{center}
      \centering \includegraphics[width=0.7\textwidth]{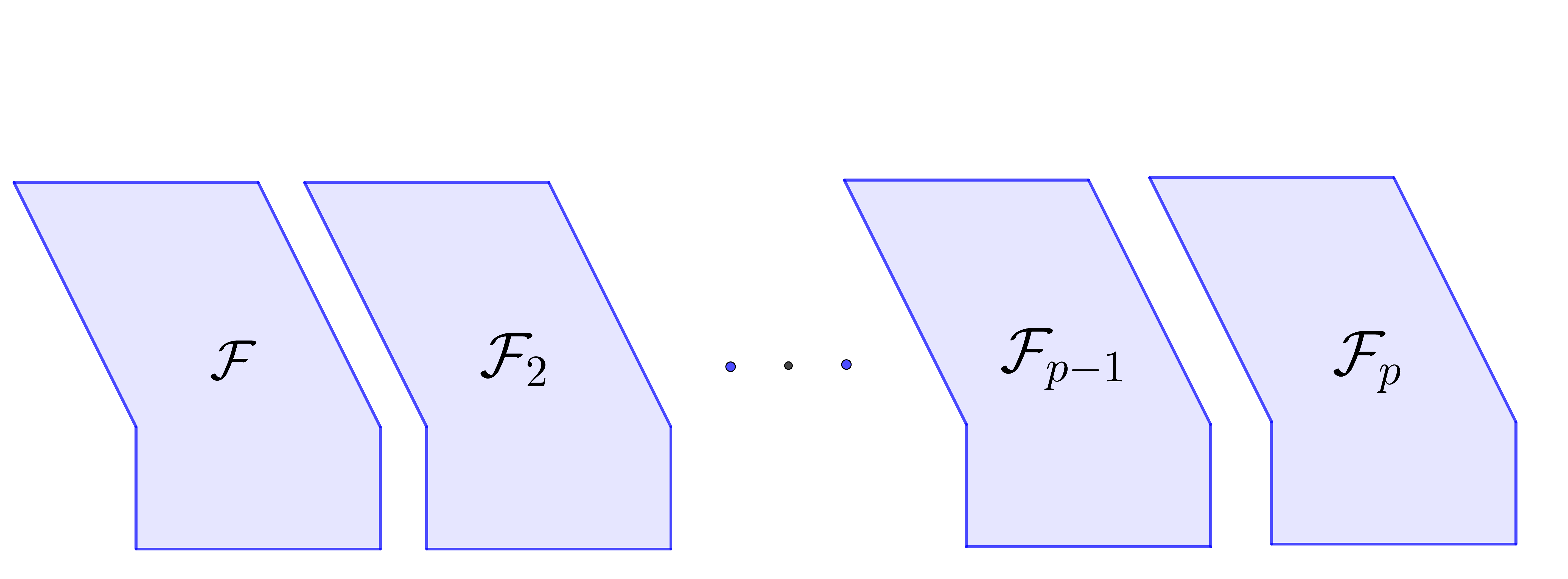}
   \end{center}
\caption{Partition of the fundamental domain of $\tau^{-p}[p]$ }\label{funddom}
\end{figure}

Following \cite{Zhu} we have the projection functor $\eta_p:\mathcal{C}_{n,p}\rightarrow \mathcal{C}_n$ which sends an object $X$ in $\mathcal{C}_{n,p}$ to an object $X$ in $\mathcal{C}_n$, and $\phi:X\rightarrow Y$ in
$\mathcal{C}_{n,p}$ to the morphism $\phi:X\rightarrow Y$ in $\mathcal{C}_{n}$.

Furthermore, if $\pi_p:\mathcal{D}\rightarrow \mathcal{C}_{n,p}$ and $\pi_1:\mathcal{D}\rightarrow \mathcal{C}_{n}$ are the  respective projections it holds that $\pi_1=\pi_p\circ\eta_p$ and the three projections are triangle functors.

After \cite[Proposition 3.3]{Zhu}, we have that $\mathcal{C}_{n,p}$ is a triangulated category with AR-triangles
 and Serre functor $\nu:=\tau[1]$ which is a fractionally Calabi-Yau category of dimension $\frac{2p}{p}$. Moreover,  $\mathcal{C}_{n,p}$ is a Krull-Schmidt category and $\mathrm{ind}(\mathcal{C}_{n,p})=\bigcup_{i=1}^{p}
     \mathrm{ind}(\mathcal{F}_i)$.

Repetitive  cluster categories were studied by Zhu in \cite{Zhu} from a purely algebraic point of view and  for type $A_n$  by Lamberti in \cite{LambertiRep} in a geometrical-combinatorial way. We will do a similar study for  type $D_n$.

\subsection{The Auslander-Reiten quiver of  $\mathcal{C}_{n,p}$}

Since  the Auslander-Reiten quiver of the
cluster category $\mathcal{C}_{n}$,  denoted by
$\Gamma(D_n,1)$,
 is the stable translation quiver built from $n$ copies
of $Q^{op}$ (see~\cite[\S1]{BMRRT},~\cite{H}), after \cite[Proposition 3.3]{Zhu}  we can see the Auslander-Reiten quiver of the repetitive  cluster category $\mathcal{C}_{n,p}$,  denoted by
$\Gamma_{n,p}$,
 as the stable translation quiver built from $np$ copies
of $Q^{op}$. The vertices of $\Gamma_{n,p}$ are:

\[ (\Gamma_{n,p})_0=\{(i,j)\mid i\in \mathbb{Z}_{np}, \ j\in Q_0\}=\mathbb{Z}_{np}\times\{0,\overline{0},1,\ldots,n-2\}; \]

and there is an arrow  $(i,j)\to (i,k)$ and an arrow $(i,k)\to (i+1,k)$ whenever there is an arrow $j\to k$ in $Q^{op}$. \\

The translation $\tau$ is given by

\[
\tau(i,j)=\left\{\begin{array}{ll}
(i-1,\overline{j}), & \text{if $i=0$, $j\in\{0,\overline{0}\}$ and
$np$ is odd; }\\
(i-1,j), & \text{otherwise.}
\end{array}\right.
\]
Where we use the convention that $\overline{\overline{0}}=0$.

We draw the quivers $\Gamma(D_n,1)$ and $\Gamma_{n,p}$ for
$n=3$ and $p=2$ as an example; see Figures~\ref{quiverd3} and~\ref{quiverd32}.
The translation $\tau$ is indicated by dotted lines
(it is directed to the left).

\begin{figure}[H]
\begin{center}

$${\small\xymatrix@-6mm{
(0,1)\ar[dr]\ar[ddr]\ar@{.}[rr] & &
 (1,1)\ar[dr]\ar[ddr]\ar@{.}[rr] & &
 (2,1)\ar[dr]\ar[ddr]\ar@{.}[rr] & &
 (0,1)\ar[dr]\ar[ddr]  \\
% (0,1)\ar[dr]\ar[ddr]\ar@{.}[r] &  \\
 & (0,0)\ar[ur]\ar@{.}[rr] & & (1,0)\ar[ur]\ar@{.}[rr] & &
 (2,0)\ar[ur]\ar@{.}[rr] & &
 (0,\overline{0})\\
% (0,\overline{0})\ar@{.}[r] &  \\
 & (0,\overline{0})\ar[uur]\ar@{.}[rr] & &
 (1,\overline{0})\ar[uur]\ar@{.}[rr]
 & & (2,\overline{0})\ar[uur]\ar@{.}[rr] & &
 (0,0)%\ar@{.}[r]  &
}}
$$
\caption{The quiver $\Gamma(D_3,1)$}
\label{quiverd3}
\end{center}
\end{figure}
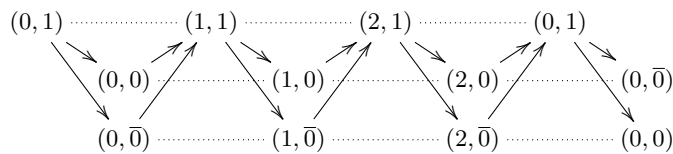

\begin{figure}[H]
\begin{center}

$${\small\xymatrix@-6mm{
&(0,1)\ar[dr]\ar[ddr]\ar@{.}[rr] & &
 (1,1)\ar[dr]\ar[ddr]\ar@{.}[rr] & &
 (2,1)\ar[dr]\ar[ddr]\ar@{.}[rr] & &
 (3,1)\ar[dr]\ar[ddr]\ar@{.}[rr] & &
 (4,1)\ar[dr]\ar[ddr]\ar@{.}[rr] & &
 (5,1)\ar[dr]\ar[ddr]\ar@{.}[rr] & &
 (0,1)\ar[dr]\ar[ddr] \\
 & & (0,0)\ar[ur]\ar@{.}[rr] & &
 (1,0)\ar[ur]\ar@{.}[rr] & &
 (2,0)\ar[ur]\ar@{.}[rr] & &
 (3,0)\ar[ur]\ar@{.}[rr] & &
 (4,0)\ar[ur]\ar@{.}[rr] & &
 (5,0)\ar[ur]\ar@{.}[rr] & &
 (0,0) \\
 & & (0,\overline{0})\ar[uur]\ar@{.}[rr] & &
 (1,\overline{0})\ar[uur]\ar@{.}[rr] & &
 (2,\overline{0})\ar[uur]\ar@{.}[rr] & &
 (3,\overline{0})\ar[uur]\ar@{.}[rr] & &
 (4,\overline{0})\ar[uur]\ar@{.}[rr] & &
 (5,\overline{0})\ar[uur]\ar@{.}[rr] & &
 (0,\overline{0})
}}
$$
\caption{The quiver $\Gamma_{3,2}$}
\label{quiverd32}
\end{center}
\end{figure}
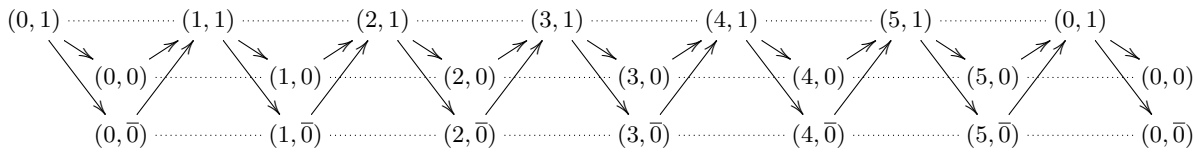

Note that the switch described above only occurs
when $np$ is odd.

\section{$\mathcal{C}_{n,p}$ and the link to the cluster category}

One of our goals is to realise the Auslander-Reiten quiver ( or AR-quiver for short) for the repetitive cluster category in terms of the AR-quiver of a cluster category of type $D_t$ for certain value of $t$. \\

In the following denote by $(\Gamma_{n,p},\tau_{n,p})$ the AR-quiver of the repetitive cluster category
$\mathcal{C}_{n,p}$ and by  $(\Gamma(D_t,1),\tau_t)$ the AR-quiver of the cluster category $\mathcal{C}_t$.

\begin{lema}\label{subquiver}

Let $t=np$. Then $(\Gamma_{n,p},\tau_{n,p})$ is a
subquiver of $(\Gamma(D_t,1),\tau_t)$ .

\end{lema}

\begin{proof} 
We define a subquiver $\tilde{\Gamma}$ of
$\Gamma(D_{t},1)$ taking the union of the
$\tau_{t}$-orbits of the  bottom $n$ rows of the quiver $\Gamma(D_{t},1)$, illustrated
in the darkest strip in figure \ref{fig.rep}. Since $\Gamma(D_{t},1)$ is a stable translation quiver, the same remains true for $\tilde{\Gamma}$.

\begin{figure}[H]
 \unitlength=5cm
   \begin{center}
      \centering \includegraphics[width=0.7\textwidth]{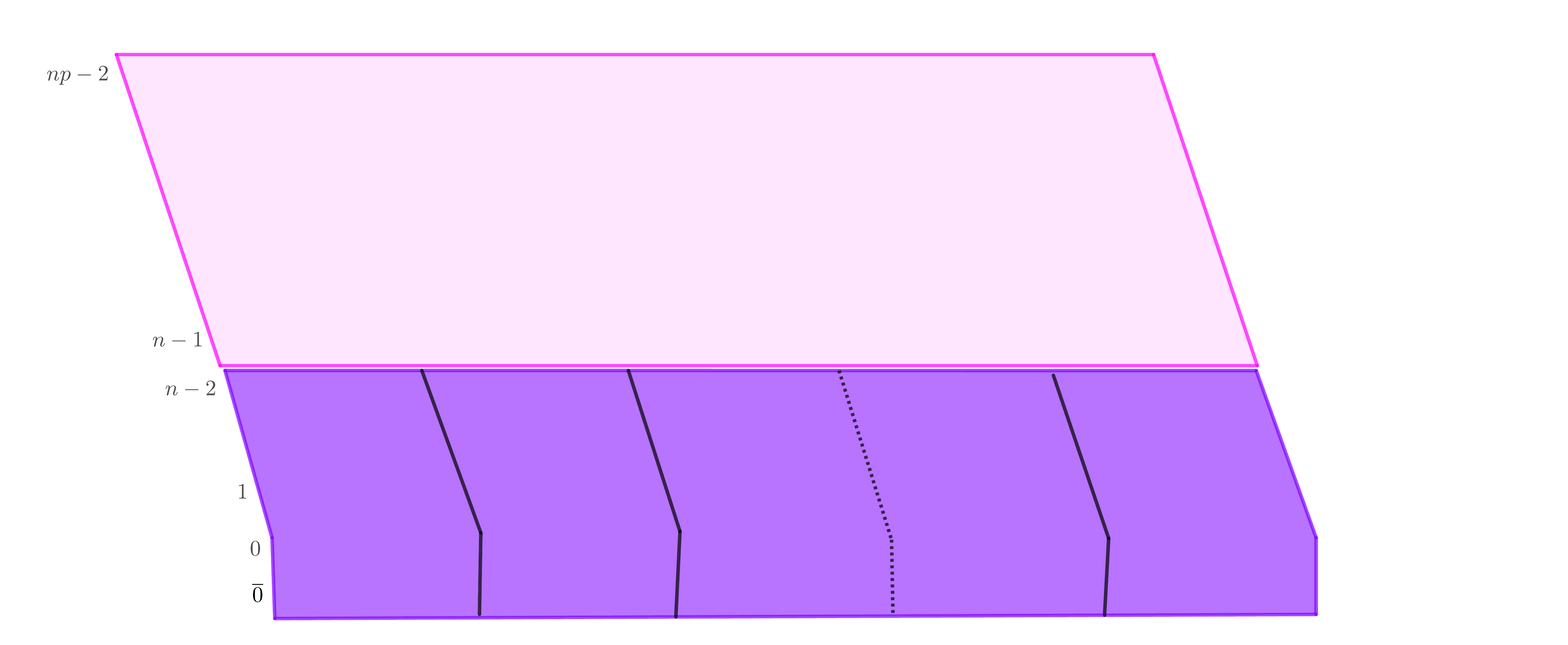}
\caption{Inclusion $\Gamma_{n,p}\subset \Gamma(D_{np},1)$.}
\label{fig.rep}
   \end{center}
\end{figure}

We will show that there is an  isomorphism of stable translation
quivers between $\Gamma_{n,p}$ and $\tilde{\Gamma}$. 

By construction, $\tilde{\Gamma}$ and  $\Gamma_{n,p}$ have both the same number of rows (namely $n$).
In order to show a quiver isomorphism between $\tilde{\Gamma}$ and  $\Gamma_{n,p}$ it remaind to 
compare the induced action of the
auto equivalence $(\tau_{n,p}^{-1}[1])^p$ on $\Gamma_{n,p}$ with the action of
$(\tau_{t}^{-1}[1])|_{\tilde{\Gamma}}$ on $\tilde{\Gamma}$.
Remember that $\Gamma(D_{r},1)$ identifies the vertices $(0,j)$ and $(r,j)$ for every $j\notin \{0,\overline{0}\}$. For $r$ even, also identifies the vertices $(0,0)$ and $(0,\overline{0})$ with the vertices $(r,0)$ and  $(r,\overline{0})$ respectively. However, for $r$ odd it identifies  $(0,0)$ with $(r,\overline{0})$ and $(0,\overline{0})$ with $(r,0)$.
Then the actions coincide because $np$ is odd if and only if $n$ and $p$ are both odd.

Finally, since the meshes of the quivers
$\tilde{\Gamma}$  and  $\Gamma_{n,p}$ coincide we deduce that this gives an
isomorphism of stable translation quivers.

\end{proof}

Now,  we describe the other component arising in the
 translation quiver $(\Gamma(D_{np},1),\tau)$.\\

\begin{prop}\label{prop:A-components}
The quiver $\Gamma(D_{np},1)$ has $1$ connected
component isomorphic to the Auslander-Reiten quiver
of $D^b(A_{n(p-1)})/\tau^{np}$.
\end{prop}
\begin{proof}
We consider the following subset $Z$ of  vertices of the quiver
 $\Gamma(D_{np},1)$:

$$Z:= \{(i,j)\mid i\in\mathbb{Z}_{np}, j\in\{n-1,\cdots ,np-2\} \}$$

Such a set $Z$ is the union of the first $n(p-1)$ orbits at the top of the quiver
$\Gamma(D_{np},1)$. It is clear that the translation quiver
generated by $Z$ (i.e. the full subquiver induced by $Z$,
together with $\tau$) is a connected component of $\Gamma(D_{np},1)$.
Each row is of length $np$. Then $Z$ is isomorphic to
the Auslander-Reiten quiver of $D^b(A_{n(p-1)})/\tau^{np}$.

%1) The set $Z$ is closed under the translation $\tau$, since,
%by definition, $Z$ is the union of all vertices of certain
%rows and $\tau^m$ shifts vertices along a row.

%
%Now by definition, $Z$ is the union of $n(p-1)$ rows, namely the
%rows $(\cdot, np-2)$, $(\cdot,np-3)$ up to $(\cdot,n-1)$.

\end{proof}

Since, by lemma~\ref{subquiver},  $\Gamma_{n,p}$ is a connected component of $\Gamma(D_{np},1)$
and the vertices of $\Gamma(D_{np},1)$
are exhausted by the vertices of $\Gamma_{n,p}$ and $Z$, we obtain the following corollary.

\begin{coro}
The  quiver $(\Gamma(D_{np},1),\tau)$
is the union of the following connected components:
\[
(\Gamma(D_{np},1),\tau_{np})=\Gamma_{n,p}\cup
 \Gamma(D^b(A_{n(p-1)})/\tau^{np}),
\]
where $\Gamma(D^b(A_{n(p-1)})/\tau^{np})$ denotes the Auslander-Reiten
quiver of  $D^b(A_{n(p-1)})/\tau^{np}$.
\end{coro}

We ilustrate on the following example the results presented.

\begin{ejem}
Let $n=4$ and $p=2$.
\end{ejem}

\begin{figure}[H]
\begin{center}
$$
\tiny{
\xymatrix@-8mm{
06 \ar[ddr] \ar@{.}[rr] && 16 \ar[ddr] \ar@{.}[rr] && 26 \ar[ddr] \ar@{.}[rr] && 36 \ar[ddr] \ar@{.}[rr] && 46 \ar[ddr] \ar@{.}[rr] &&
56 \ar[ddr] \ar@{.}[rr] && 66 \ar[ddr] \ar@{.}[rr] && 76 \ar[ddr] \ar@{.}[rr] && 06 \ar[ddr] \\ \\
& 05 \ar[ddr] \ar[uur] \ar@{.}[rr] && 15 \ar[ddr] \ar[uur] \ar@{.}[rr] && 25 \ar[ddr] \ar[uur] \ar@{.}[rr] && 35 \ar[ddr]
\ar[uur] \ar@{.}[rr] && 45 \ar[ddr] \ar[uur] \ar@{.}[rr] && 55 \ar[ddr] \ar[uur] \ar@{.}[rr] && 65\ar[ddr] \ar[uur] \ar@{.}[rr] && 75 \ar[ddr] \ar[uur] \ar@{.}[rr] &&  05 \ar[ddr] \\ \\
&& 04 \ar[ddr] \ar[uur] \ar@{.}[rr] && 14 \ar[ddr] \ar[uur] \ar@{.}[rr] && 24 \ar[ddr] \ar[uur] \ar@{.}[rr] && 34 \ar[ddr]
\ar[uur] \ar@{.}[rr] && 44 \ar[ddr] \ar[uur] \ar@{.}[rr] && 54 \ar[ddr] \ar[uur] \ar@{.}[rr] && 64\ar[ddr] \ar[uur] \ar@{.}[rr] && 74 \ar[ddr] \ar[uur] \ar@{.}[rr] &&  04 \ar[ddr] \\ \\
&&& 03 \ar[ddr] \ar[uur] \ar@{.}[rr] && 13 \ar[ddr] \ar[uur] \ar@{.}[rr] && 23 \ar[ddr] \ar[uur] \ar@{.}[rr] && 33 \ar[ddr]
\ar[uur] \ar@{.}[rr] && 43 \ar[ddr] \ar[uur] \ar@{.}[rr] &&
53 \ar[ddr] \ar[uur] \ar@{.}[rr] && 63  \ar[ddr] \ar[uur] \ar@{.}[rr] && 73 \ar[ddr] \ar[uur] \ar@{.}[rr] && 03 \ar[ddr] \\ \\
&&&& 02 \ar[ddr] \ar[uur] \ar@{.}[rr] && 12 \ar[ddr] \ar[uur] \ar@{.}[rr] && 22 \ar[ddr] \ar[uur] \ar@{.}[rr] && 32 \ar[ddr]
\ar[uur] \ar@{.}[rr] && 42 \ar[ddr] \ar[uur] \ar@{.}[rr] &&
52 \ar[ddr] \ar[uur] \ar@{.}[rr] && 62 \ar[ddr] \ar[uur] \ar@{.}[rr] && 72 \ar[ddr] \ar[uur] \ar@{.}[rr] && 02 \ar[ddr] \\ \\
&&&&& 01 \ar[ddr] \ar[ddddr] \ar[uur] \ar@{.}[rr] && 11 \ar[ddr] \ar[ddddr] \ar[uur] \ar@{.}[rr] && 21 \ar[ddr] \ar[ddddr] \ar[uur] \ar@{.}[rr] && 31 \ar[ddr] \ar[ddddr] \ar[uur] \ar@{.}[rr] && 41 \ar[ddr] \ar[ddddr] \ar[uur] \ar@{.}[rr] &&
51 \ar[ddr] \ar[ddddr] \ar[uur] \ar@{.}[rr] && 61 \ar[ddr] \ar[ddddr] \ar[uur] \ar@{.}[rr] &&  71 \ar[ddr] \ar[ddddr] \ar[uur] \ar@{.}[rr] && 01 \ar[ddr] \ar[ddr] \ar[ddddr] \\ \\
&&&&&& 00 \ar[uur] \ar@{.}[rr] && 10 \ar[uur] \ar@{.}[rr] && 20 \ar[uur] \ar@{.}[rr] && 30 \ar[uur] \ar@{.}[rr] && 40 \ar[uur] \ar@{.}[rr] &&
50 \ar[uur] \ar@{.}[rr] && 60 \ar[uur] \ar@{.}[rr] && 70 \ar[uur] \ar@{.}[rr]&& 00 \\ \\
&&&&&& 0\overline{0} \ar[uuuur] \ar@{.}[rr] && 1\overline{0} \ar[uuuur] \ar@{.}[rr] && 2\overline{0} \ar[uuuur] \ar@{.}[rr] && 3\overline{0} \ar[uuuur] \ar@{.}[rr] && 4\overline{0} \ar[uuuur] \ar@{.}[rr] &&
5\overline{0} \ar[uuuur] \ar@{.}[rr] && 6\overline{0} \ar[uuuur] \ar@{.}[rr] && 7\overline{0} \ar[uuuur] \ar@{.}[rr]&& 0\overline{0}
}}
$$
\caption{The translation quiver $\Gamma(D_8,1)$}
\label{fi:d7usual}
\end{center}
\end{figure}
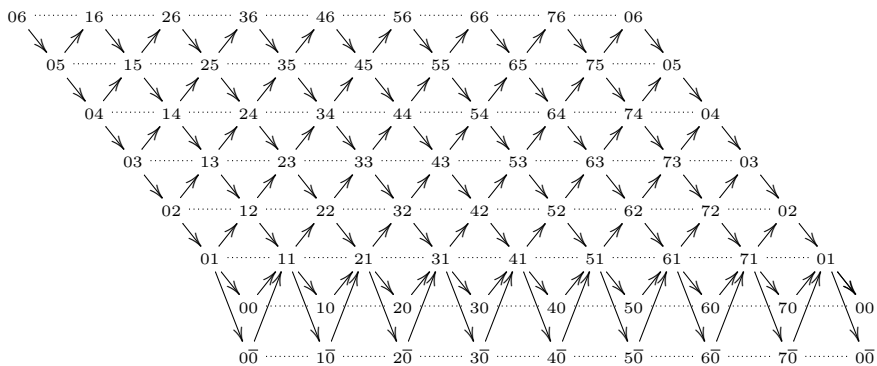

\begin{figure}[H]
\begin{center}
$$
\tiny{
\xymatrix@-8mm{
&&&& 02 \ar[ddr]  \ar@{.}[rr] && 12 \ar[ddr]  \ar@{.}[rr] && 22 \ar[ddr] \ar@{.}[rr] && 32 \ar[ddr]
 \ar@{.}[rr] && 42 \ar[ddr]  \ar@{.}[rr] &&
52 \ar[ddr]  \ar@{.}[rr] && 62 \ar[ddr]  \ar@{.}[rr] && 72 \ar[ddr]  \ar@{.}[rr] && 02 \ar[ddr] \\ \\
&&&&& 01 \ar[ddr] \ar[ddddr] \ar[uur] \ar@{.}[rr] && 11 \ar[ddr] \ar[ddddr] \ar[uur] \ar@{.}[rr] && 21 \ar[ddr] \ar[ddddr] \ar[uur] \ar@{.}[rr] && 31 \ar[ddr] \ar[ddddr] \ar[uur] \ar@{.}[rr] && 41 \ar[ddr] \ar[ddddr] \ar[uur] \ar@{.}[rr] &&
51 \ar[ddr] \ar[ddddr] \ar[uur] \ar@{.}[rr] && 61 \ar[ddr] \ar[ddddr] \ar[uur] \ar@{.}[rr] &&  71 \ar[ddr] \ar[ddddr] \ar[uur] \ar@{.}[rr] && 01 \ar[ddr] \ar[ddr] \ar[ddddr] \\ \\
&&&&&& 00 \ar[uur] \ar@{.}[rr] && 10 \ar[uur] \ar@{.}[rr] && 20 \ar[uur] \ar@{.}[rr] && 30 \ar[uur] \ar@{.}[rr] && 40 \ar[uur] \ar@{.}[rr] &&
50 \ar[uur] \ar@{.}[rr] && 60 \ar[uur] \ar@{.}[rr] && 70 \ar[uur] \ar@{.}[rr]&& 00 \\ \\
&&&&&& 0\overline{0} \ar[uuuur] \ar@{.}[rr] && 1\overline{0} \ar[uuuur] \ar@{.}[rr] && 2\overline{0} \ar[uuuur] \ar@{.}[rr] && 3\overline{0} \ar[uuuur] \ar@{.}[rr] && 4\overline{0} \ar[uuuur] \ar@{.}[rr] &&
5\overline{0} \ar[uuuur] \ar@{.}[rr] && 6\overline{0} \ar[uuuur] \ar@{.}[rr] && 7\overline{0} \ar[uuuur] \ar@{.}[rr]&& 0\overline{0}
}}
$$
\caption{The connected component $\tilde{\Gamma}_8$ isomorphic to  $\Gamma_{4,2}$}
\end{center}
\end{figure}
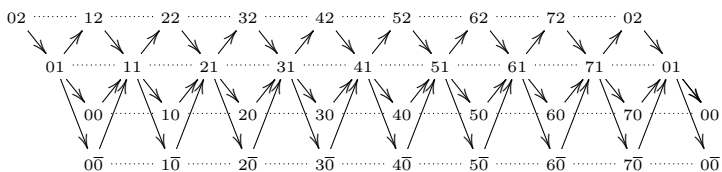

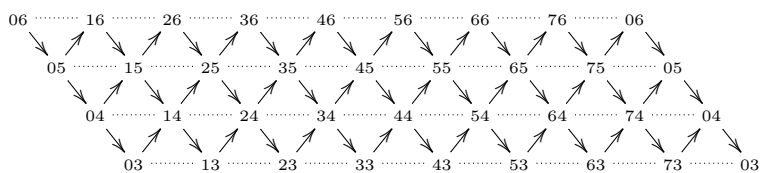
\begin{figure}[H]
\begin{center}
$$
\tiny{
\xymatrix@-8mm{
06 \ar[ddr] \ar@{.}[rr] && 16 \ar[ddr] \ar@{.}[rr] && 26 \ar[ddr] \ar@{.}[rr] && 36 \ar[ddr] \ar@{.}[rr] && 46 \ar[ddr] \ar@{.}[rr] &&
56 \ar[ddr] \ar@{.}[rr] && 66 \ar[ddr] \ar@{.}[rr] && 76 \ar[ddr] \ar@{.}[rr] && 06 \ar[ddr] \\ \\
& 05 \ar[ddr] \ar[uur] \ar@{.}[rr] && 15 \ar[ddr] \ar[uur] \ar@{.}[rr] && 25 \ar[ddr] \ar[uur] \ar@{.}[rr] && 35 \ar[ddr]
\ar[uur] \ar@{.}[rr] && 45 \ar[ddr] \ar[uur] \ar@{.}[rr] && 55 \ar[ddr] \ar[uur] \ar@{.}[rr] && 65\ar[ddr] \ar[uur] \ar@{.}[rr] && 75 \ar[ddr] \ar[uur] \ar@{.}[rr] &&  05 \ar[ddr] \\ \\
&& 04 \ar[ddr] \ar[uur] \ar@{.}[rr] && 14 \ar[ddr] \ar[uur] \ar@{.}[rr] && 24 \ar[ddr] \ar[uur] \ar@{.}[rr] && 34 \ar[ddr]
\ar[uur] \ar@{.}[rr] && 44 \ar[ddr] \ar[uur] \ar@{.}[rr] && 54 \ar[ddr] \ar[uur] \ar@{.}[rr] && 64\ar[ddr] \ar[uur] \ar@{.}[rr] && 74 \ar[ddr] \ar[uur] \ar@{.}[rr] &&  04 \ar[ddr] \\ \\
&&& 03  \ar[uur] \ar@{.}[rr] && 13  \ar[uur] \ar@{.}[rr] && 23  \ar[uur] \ar@{.}[rr] && 33
\ar[uur] \ar@{.}[rr] && 43 \ar[uur] \ar@{.}[rr] &&
53  \ar[uur] \ar@{.}[rr] && 63   \ar[uur] \ar@{.}[rr] && 73  \ar[uur] \ar@{.}[rr] && 03
}}
$$
\caption{The connected component isomorphic to $\Gamma(D^b(A_{4})/\tau^{8})$}
\label{fi:d7usual}
\end{center}
\end{figure}

%*****observar que podemos incluir a A3 isomorfo a D3 para cualquier p en D3p, incluso para p=2 .\\

%\subsection{Triangulated equivalence for $\mathcal{C}_{n,p}$ }
%Here we desire to compare the category $\mathcal{C}_{n,p}$ with the cluster category of
%type $D_t$ for  $t=np$.

Observe that the  inclusion of the AR-quiver of $\mathcal{C}_{n,p}$ in the AR-quiver of the cluster category
$\mathcal{C}_{t}$ for $t=np$, given in lemma \ref{subquiver}, does not give rise to an
inclusion at the level of full subcategories.  In  particular  $\mathcal{C}_{n,p}$  is not a triangulated subcategory of $\mathcal{C}_{np}$. However, it is possible to prove
that it is triangulated equivalent to a quotient category of $\mathcal{C}_{np}$ in the sense of J\o{}rgensen \cite{Jorgensen}.

We now recall the definition of quotient categories and some  properties.
Let $\mathcal{C}$ be an additive category and $\X$ a class of objects of
$\mathcal{C}$. Then the \emph{quotient category} $\mathcal{C}_{\X}$ has by definition the
same objects as $\mathcal{C}$, but the morphism spaces are taken modulo all the
morphisms factoring through an object of $\X$.

If $\mathcal{C}$ is a triangulated category, then $\mathcal{C}_{\X}$
 is not triangulated for all choices of $\X$. However,  $\mathcal{C}_{\X}$ is
always pre-triangulated (\cite[Theorem 2.2]{Jorgensen}) and taking a particular choice of the
class $\X$, $\mathcal{C}_{\X}$ becomes a triangulated category (\cite[ Theorem 3.3 ]{Jorgensen}).

The following proposition  follows from a similar reasoning to that of the proof
of \cite[ Proposition 4.2 ]{LambertiRep} taking
$\X$ by the additive full subcategory generated by the indecomposable objects in the first $n(p-1)$ orbits at the top of the quiver $\Gamma(D_{np},1)$ and using the results of J\o{}rgensen in \cite{Jorgensen} for type $D$ instead of those for type $A$.

\begin{prop}\label{quot}
The repetitive cluster category $\mathcal{C}_{n,p}$ is triangulated equivalent to a quotient of the
cluster category  $\mathcal{C}_t$ for $t=np$.
\end{prop}

\section{Geometric model of $\mathcal{C}_{n,p}$ }

In this section we present the geometric model for $\mathcal{C}_{n,p}$. It is a simple modification of the model constructed by Schiffler in \cite{S} for the cluster category of type $D_n$. We follow the notation of \cite{S}.

\subsection{Tagged edges} \label{edges}

Let $n\ge 3$ and $p\ge 1$. Consider a regular polygon $\textsf{P}_{np}$ with $np$ vertices and one puncture in its
center. We label the vertices of $\textsf{P}_{np}$ counterclockwise $1,2,\cdots, np$.

If $a, b$ are any two vertices on the boundary, let $\delta_{a,b}$ denote
the path in the counterclockwise
direction from 	$a$ to $b$ along the boundary of  $\textsf{P}_{np}$.
If $a=b$ the path runs around the polygon exactly once.  Let $|\delta_{a,b}|$
be the number of vertices on
the path $\delta_{a,b}$ (including $a$ and $b$).

\begin{figure}[H]
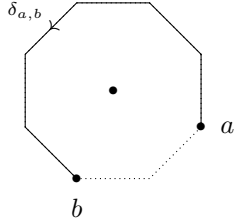

\begin{center}

\[
\xy/r3pc/:{\xypolygon8"A"{~<{}~>{.}{}}}
*+{{\scriptstyle \bullet}},
\POS"A6"\drop{\scriptstyle\bullet}
\POS"A6"\drop{\begin{array}{c}\\ \\ b\end{array}}
\POS"A8"\drop{\scriptstyle\bullet}
\POS"A8"\drop{\qquad a}
\POS"A8" \ar@{-}   "A1"
\POS"A1" \ar@{-}   "A2"
\POS"A2" \ar@{-}   "A3"
\POS"A3" \ar|-{\SelectTips{cm}{}\object@{>}}@{-}_{\zd_{a,b}}   "A4"
\POS"A4" \ar@{-}   "A5"
\POS"A5" \ar@{-}   "A6"
\endxy   \]
\caption{The path $\delta_{a,b}$ on the punctured polygon $\textsf{P}_{np}$.}
\label{Pn}
\end{center}
\end{figure}

An {\em edge} $\alpha$ from $a$ to $b$ is a  path
lying in the interior of the punctured polygon $\textsf{P}_{np}$,
with no self crossings and homotopic to $\delta_{a,b}$ for $|\delta_{a,b}| \geq np-n+3$.

We say that two edges are equivalent if they start in the same vertex, end in the same
vertex and they are homotopic. We will denote $M_{a,b}$  the
equivalence class of edges homotopic to $\delta_{a,b}$.

A \emph{tagged edge} $M_{a,b}^\ze$  is an edge $ M_{a,b}$ with a tag $\ze=\pm 1$ in such a way that if
$a\ne b$, then $\ze=1$.
As in \cite{S}, denote by $\edges $ the set of {\em tagged edges}.

If $a=b$, there are exactly two tagged edges $M_{a,a}^{-1}$ and $M_{a,a}^{1}$ and if $a\ne b$ there is
exactly one tagged edge $M_{a,b}^1$.
If the tag is $1$ and  $a\ne b$, we will often drop the exponent and write
$M_{a,b}$ instead of $M_{a,b}^1$.

A simple count shows that there are $ pn^2$
elements in $\edges$. These tagged edges will correspond to the
indecomposable objects in the repetitive cluster category.

\subsection{Elementary moves}\label{sect elementary moves}
Adapting a concept from \cite{S}, we will  define elementary
moves, which will correspond to irreducible morphisms in the repetitive cluster
  category.

An elementary move  sends a tagged edge  $M_{a,b}^\ze\in \edges$ to
  another tagged edge $M_{a',b'}^{\ze'}\in\edges$ in the following way:

\begin{enumerate}
\item If $| \zd_{a,b}|  =np-n+3$, then there is precisely one elementary
  move    $M_{a,b}\mapsto M_{a,b+1}$.

\item If $ np-n+4 \le | \zd_{a,b}| \le np-1$, then
there are precisely two elementary
  moves  $M_{a,b}\mapsto M_{a+1,b}$ and $M_{a,b}\mapsto M_{a,b+1}$.

\item If $| \zd_{a,b}|  =np$, then
there are precisely three elementary
  moves  $M_{a,b}\mapsto M_{a+1,b}$, $M_{a,b}\mapsto M_{a,a}^1$ and
  $M_{a,b}\mapsto  M_{a,a}^{-1}$.

\item If $| \zd_{a,b}|    =np+1$, then $a=b$ and there is precisely one elementary move
  $M_{a,a}^\ze\mapsto M_{a+1,a}$.
\end{enumerate}

\begin{nota} Observe that when we write $M_{a,b}$ the indices $a,b$ have to be taken modulo $np$.
\end{nota}

\begin{figure}[H]
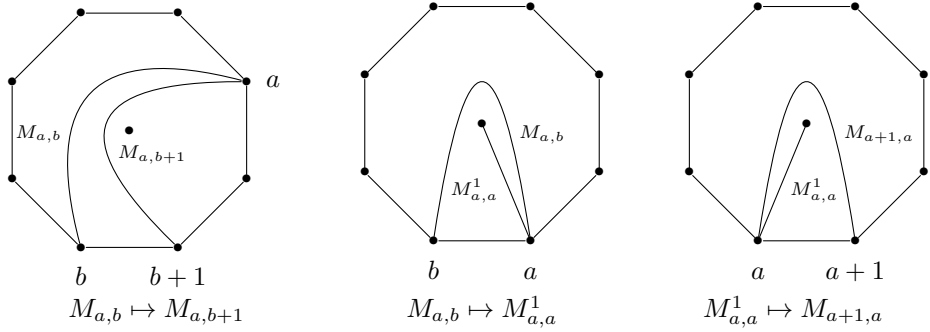

\[\begin{array}{ccc}
\xy/r4pc/:{\xypolygon8"A"{~<{}~>{-}{\scriptstyle \bullet}}}
*+{\scriptstyle\bullet},
\POS"A1"\drop{\qquad a}
\POS"A7"\drop{\begin{array}{c}\\ \\ b+1 \end{array}  }
\POS"A6"\drop{\begin{array}{c}\\ \\b \end{array}  }
\POS"A6" \ar@{}^(.3){M_{a,b}}@/^9ex/ "A1"
\POS"A7" \ar@{}_(.3){M_{a,b+1}}@/^10ex/ "A1",
\endxy
\quad
&\quad
\xy/r4pc/:{\xypolygon8"A"{~<{}~>{-}{\scriptstyle \bullet}}}
*+{\scriptstyle\bullet},
\POS"A7"\drop{\begin{array}{c}\\ \\a \end{array}  }
\POS"A6"\drop{\begin{array}{c}\\ \\b \end{array}  }
\POS"A6" \ar@{}^(.8){M_{a,b}}@/^14ex/ "A7"
\POS"A7" \ar@{}^(.5){M_{a,a}^1} @/^0ex/ "A0",
\endxy
\quad
&\quad
\xy/r4pc/:{\xypolygon8"A"{~<{}~>{-}{\scriptstyle \bullet}}}
*+{\scriptstyle\bullet},
\POS"A6"\drop{\begin{array}{c}\\ \\ a \end{array}  }
\POS"A7"\drop{\begin{array}{c}\\ \\ a+1 \end{array}  }
\POS"A6" \ar@{}^(.8){M_{a+1,a}}@/^14ex/ "A7"
\POS"A6" \ar@{}_(.5){M_{a,a}^1} @/^0ex/ "A0",
\endxy \\
M_{a,b}\mapsto M_{a,b+1}
&M_{a,b}\mapsto M_{a,a}^1
&M_{a,a}^1\mapsto M_{a+1,a}
\end{array}\]
\caption{Examples of elementary moves.}\label{fig elementary moves}
\end{figure}

\subsection{Translation} \label{sect tau}
We define the {\em translation} $\tau$ to be the following  bijection
$\tau:\edges\to\edges$:

\begin{enumerate}
\item If $a\ne b$ then $\tau  M_{a,b}=M_{a-1,b-1}$.
\item If $a= b$ then  $\tau  M_{a,a}^\ze=M_{a-1,a-1}^{-\ze}$, for
  $\ze=\pm 1$.
\end{enumerate}

\vspace*{.5cm}

Follows directly from the definition of $\tau$ that for  $a\ne b$,
$\tau^{np} \,M_{a,b}^\ze = M_{a,b}^\ze$ regardless the parity of $np$. For $a=b$, if $np$ is even, $\tau^{np} \,M_{a,a}^\ze = M_{a,a}^\ze$ and if $np$ is odd, $\tau^{np}$ switches the tags, that is;  $\tau^{np} \,M_{a,a}^\ze = M_{a,a}^{-\ze}$.

Observing that our tagged edges are a subset of the set of tagged edges defined in \cite{S}. Precisely, the ones with  $| \zd_{a,b}|   \ge np-n+3$ instead of $|\zd_{a,b}|\ge 3$. For $p=1$ we have exactly the definition of \cite{S}.
In particular, we have the following lemma.

\begin{lema}\cite[Lemma 3.6]{S}\label{mov elementales}
Let $M_{a,b}^\zl,\ M_{c,d}^{\ze}$ be two tagged edges. Then there is an elementary move
$M_{a,b}^{\zl} \mapsto M_{c,d}^{\ze} $ if and only if there is an elementary
move
$\tau  M_{c,d}^{\ze} \mapsto M_{a,b}^{\zl}$.
\end{lema}

\subsection{Quiver of tagged edges  of $\textsf{P}_{np}$}
As next we associate a translation quiver $\Gamma_{\odot}$
to the tagged edges of $\textsf{P}_{np}$ with the intention of modelling the AR-quiver of the
category $\mathcal{C}_{n,p}$.

\begin{defi} Let $\Gamma_{\odot}$ be the quiver whose vertices are the tagged edges $M\in\edges$  on the punctured polygon  $\textsf{P}_{np}$. Given $M,N\in\edges$ there is an arrow $M\to
 N$ in $\Gamma_{\odot}$ whenever there is an elementary move
$M\mapsto N$.

\end{defi}

Note that $\Gamma_{\odot} $ has no loops and no  multiple arrows.\\

\begin{lema}
The pair $(\Gamma_{\odot}, \tau)$ is a stable translation quiver.
\end{lema}

\begin{proof}
  Clearly the map $\tau$ is bijective.
  As $\Gamma_{\odot}$ is
finite, we only need to persuade us that the number of arrows from a tagged edge  $M$ to a tagged edge
$N$ is equal to the number of arrows from $\tau N$ to $M$. As there is at most one
arrow between any two tagged edges, we only have to check that there is an arrow from
$M$ to $N$ if and only if there is an arrow from $\tau N$ to $M$. It follows directly from Lemma \ref{mov elementales}.
\end{proof}

\subsection{The category of tagged edges $\mathcal{C}(\textsf{P}_{np})$}\label{cat of tagged edges}

We will now define a $\K$-linear additive category of tagged edges $\mathcal{C}(\textsf{P}_{np})$ as
the mesh category $\mesh (\Gamma_{\odot},\tau)$ of $(\Gamma_{\odot}, \tau)$ (as in section \ref{mesh cat and mesh relations}). More specifically, the objects are direct sums of tagged edges in $\edges$. The set of morphisms from a tagged edge $Y$ to a tagged edge $X$ is the quotient of the vector space over $\K$
spanned by sequences of elementary moves from $Y$ to $X$ by the
subspace generated by the  mesh relations

\[ m_X=\sum_{Y\stackrel{\za}{\rightarrow} X} \tau
X\stackrel{\zs(\za)}{\rightarrow} Y \stackrel{\za}{\rightarrow} X.\]

\section{Equivalence of categories}\label{equivalencia de cat}
In this section, we will prove the equivalence between the category $\mathcal{C}(\textsf{P}_{np})$
and the repetitive cluster category $\mathcal{C}_{n,p}$.

\begin{teo} \label{isomorfismo de quivers}
The quiver $(\Gamma_{\odot}, \tau)$ is a translation quiver isomorphic
to the Auslander-Reiten quiver $\Gamma_{n,p}$ of  ${\mathcal C}_{n,p}$.
\end{teo}

\begin{proof}
Consider the morphism  $\phi_p: \Gamma_{n,p}\rightarrow \Gamma_{\odot}$ such that

%sends a vertice $$(i,j) \longmapsto M_{i+1,i+np+1-j}^{\ze}$$ if $j\notin\{0,\overline{0}\}$ and

\[
\phi_p(i,j)=\left\{\begin{array}{lll}
M_{i+1,i+np+1-j}, & \text{if  $j\notin\{0,\overline{0}\}$; }\\
M_{i+1,i+1}, & \text{if  $j=0$ and
$i$ is even; }\\
M_{i+1,i+1}, & \text{if  $j=\overline{0}$ and
$i$ is odd; }\\
M_{i+1,i+1}^{-1}, & \text{otherwise.}
\end{array}\right.
\]

It is clear that $\phi_p$ is a bijection between the vertices of both quivers that sends the $\tau$-orbit of the vertex $(0,j)$ to the $\tau$-orbit of the vertex $M_{1,np+1-j}$; the $\tau$-orbit of the vertex $(0,0)$ to the $\tau$-orbit of the vertex $M_{1,1}$; and  the $\tau$-orbit of the vertex $(0,\overline{0})$ to the $\tau$-orbit of the vertex $M_{1,1}^{-1}$. Moreover, the arrows $(i,j)\rightarrow (i',j')$ agrees with the elementary moves $\phi_p(i,j)\mapsto \phi_p(i',j')$.

\end{proof}

Since $\mathcal{C}(\textsf{P}_{np})$ is
the mesh category $\mesh (\Gamma_{\odot},\tau)$ of $(\Gamma_{\odot}, \tau)$ we obtain the following corollary.

\begin{coro}\label{main result}
The repetitive cluster category of type $D_n$ is equivalent to the category of tagged edges $\mathcal{C}(\textsf{P}_{np})$.
\end{coro}
\qed

Observe that the fundamental domain $\mathcal{F_1}$ of the cluster category $\mathcal{C}_n$ is in correspondence  (via $\phi_p$ ) with the tagged edges $M_{a,b}^\ze$ for $a\in \{1,\cdots, n\}$; and in general $\mathcal{F}_k$, the $F^k$-shift of $F$,  is in correspondence  (via $\phi_p$ ) with the tagged edges $M_{a,b}^\ze$ for $a\in \{(k-1)n+1,\cdots, kn\}$. Then the action of $F$ on $\mathcal{C}_{n,p}$ can be see as a counterclockwise  rotation $\rho$ through $\frac{2\pi}{p}$ around the center of $\textsf{P}_{np}$.

Given a tagged edge $M_{a,b}^\ze\in \textsf{P}_{np}$, we can  identify the vertices $a+1$ and   $a+1+n(p-1)$ and delete all the edges between. This gives us a new tagged edge  $\mu_p(M_{a,b}^\ze)\in \textsf{P}_{n}$. The indices $a,b$ have to be taken modulo $np$ on the punctured polygon $\textsf{P}_{np}$ and modulo $n$ on the punctured polygon $\textsf{P}_{n}$.

\begin{figure}[H]
\[\begin{array}{cccc}
\quad
\xy/r4pc/:{\xypolygon8"A"{~<{}~>{-}{\scriptstyle \bullet}}}
*+{\scriptstyle\bullet},
\POS"A7"\drop{\begin{array}{c}\\ \\1 \end{array}  }
\POS"A6"\drop{\begin{array}{c}\\ \\8 \end{array}  }
\POS"A2"\drop{\begin{array}{c}4\\ \\ \end{array}  }
\POS"A3"\drop{\begin{array}{c}5 \\ \\ \end{array}  }
\POS"A2" \ar@{}^(.9){M_{4,5}}@/^14ex/ "A3"
\POS"A6" \ar@{}^(.8){M_{1,8}}@/^14ex/ "A7",
\endxy
\quad& \stackrel{\mu_2}{\longmapsto}
&\quad
\xy/r4pc/:{\xypolygon4"A"{~<{}~>{-}{\scriptstyle \bullet}}}
*+{\scriptstyle\bullet},
\POS"A4"\drop{\begin{array}{c}\\ \\ 1 \end{array}  }
\POS"A3"\drop{\begin{array}{c}\\ \\ 4 \end{array}  }
\POS"A3" \ar@{}_(.5){M_{1,4}}@/^14ex/ "A4",
\endxy
\end{array}\]
\caption{Example of $\mu_2:\textsf{P}_{8} \rightarrow \textsf{P}_{4}$.}\label{mu_p}
\end{figure}

It follows that this projection $\mu_p: \textsf{P}_{np} \rightarrow \textsf{P}_{n}$  corresponds to
the projection functor $\eta_p: {\mathcal C}_{n,p}\twoheadrightarrow {\mathcal C}_{n}$.

Moreover we have the following commutative diagram

\[   \xymatrix{   {\mathcal C}_{n,p}  \ar[r]^{\eta_p} \ar[d]_{\phi_p}  &  {\mathcal C}_{n} \ar[d]^{\phi_1} \\
                    \mathcal{C}(\textsf{P}_{np}) \ar[r]_{\mu_p}  &   \mathcal{C}(\textsf{P}_{n})     } \]

\begin{ejem}

We ilustrate the  Auslander-Reiten quiver of the category $\mathcal{C}(\textsf{P}_{np})$ for  $p=2$ and $n=3$. The
translation $\tau$ is indicated by dotted lines (it is directed to the left).

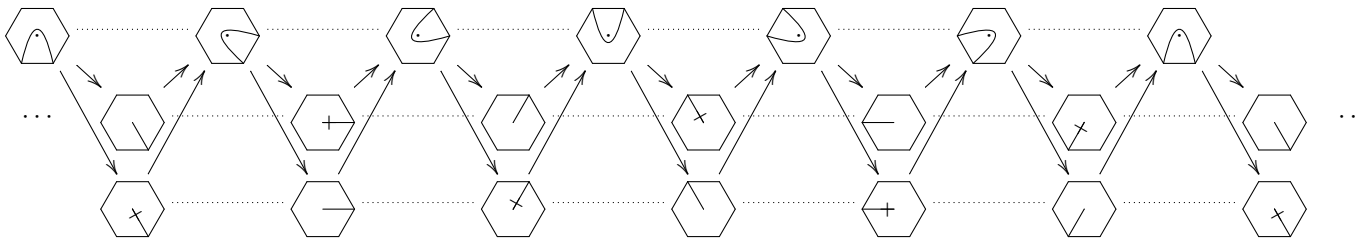
\begin{figure}[H]
\[
\xymatrix@R=6pt@C=6pt{
\xy/r1pc/:{\xypolygon6"A"{~<{}~>{-}{}}}
*+{\cdot},
\POS"A5" \ar@{-}@/^{3ex}/   "A6"
\endxy
\ar[rd]\ar[rdd]\ar@{.}[rr]
&&
\xy/r1pc/:{\xypolygon6"A"{~<{}~>{-}{}}}
*+{\cdot},
\POS"A6" \ar@{-}@/^{3ex}/    "A1"
\endxy
\ar[rd]\ar[rdd]\ar@{.}[rr]
&&
\xy/r1pc/:{\xypolygon6"A"{~<{}~>{-}{}}}
*+{\cdot},
\POS"A1" \ar@{-}@/^{3ex}/   "A2"
\endxy
\ar[rd]\ar[rdd]\ar@{.}[rr]
&&\xy/r1pc/:{\xypolygon6"A"{~<{}~>{-}{}}}
*+{\cdot},
\POS"A2" \ar@{-}@/^{3ex}/   "A3"
\endxy
\ar[rd]\ar[rdd]\ar@{.}[rr]
&&\xy/r1pc/:{\xypolygon6"A"{~<{}~>{-}{}}}
*+{\cdot},
\POS"A3" \ar@{-}@/^{3ex}/   "A4"
\endxy
\ar[rd]\ar[rdd]\ar@{.}[rr]
&&\xy/r1pc/:{\xypolygon6"A"{~<{}~>{-}{}}}
*+{\cdot},
\POS"A4" \ar@{-}@/^{3ex}/   "A5"
\endxy
\ar[rd]\ar[rdd]\ar@{.}[rr]
&&\xy/r1pc/:{\xypolygon6"A"{~<{}~>{-}{}}}
*+{\cdot},
\POS"A5" \ar@{-}@/^{3ex}/   "A6"
\endxy
\ar[rd]\ar[rdd] &
\\
\cdots &\xy/r1pc/:{\xypolygon6"A"{~<{}~>{-}{}}}
\POS"A6" \ar@{-}   "A0"
\endxy
\ar[ru]\ar@{.}[rr]&&
\xy/r1pc/:{\xypolygon6"A"{~<{}~>{-}{}}}
\POS"A1" \ar|-(0.8){\SelectTips{cm}{}\object@{+}}@{-}    "A0"
\endxy
\ar[ru] \ar@{.}[rr]&&
\xy/r1pc/:{\xypolygon6"A"{~<{}~>{-}{}}}
\POS"A2" \ar@{-}   "A0"
\endxy
\ar[ru]\ar@{.}[rr]&&
\xy/r1pc/:{\xypolygon6"A"{~<{}~>{-}{}}}
\POS"A3" \ar|-(0.8){\SelectTips{cm}{}\object@{+}}@{-}   "A0"
\endxy
\ar[ru]\ar@{.}[rr]
&&\xy/r1pc/:{\xypolygon6"A"{~<{}~>{-}{}}}
\POS"A4" \ar@{-}   "A0"
\endxy
\ar[ru]\ar@{.}[rr]
&&\xy/r1pc/:{\xypolygon6"A"{~<{}~>{-}{}}}
\POS"A5" \ar|-(0.8){\SelectTips{cm}{}\object@{+}}@{-}    "A0"
\endxy
\ar[ru]\ar@{.}[rr]
&&\xy/r1pc/:{\xypolygon6"A"{~<{}~>{-}{}}}
\POS"A6" \ar@{-}   "A0"
\endxy &\cdots
\\
&\xy/r1pc/:{\xypolygon6"A"{~<{}~>{-}{}}}
\POS"A6" \ar|-(0.8){\SelectTips{cm}{}\object@{+}}@{-} "A0"
\endxy
\ar[ruu]\ar@{.}[rr]&&
\xy/r1pc/:{\xypolygon6"A"{~<{}~>{-}{}}}
\POS"A1" \ar@{-}   "A0"
\endxy
\ar[ruu]\ar@{.}[rr]&&
\xy/r1pc/:{\xypolygon6"A"{~<{}~>{-}{}}}
\POS"A2" \ar|-(0.8){\SelectTips{cm}{}\object@{+}}@{-}  "A0"
\endxy
\ar[ruu]\ar@{.}[rr]&&
\xy/r1pc/:{\xypolygon6"A"{~<{}~>{-}{}}}
\POS"A3"  \ar@{-}  "A0"
\endxy
\ar[ruu]\ar@{.}[rr]
&&\xy/r1pc/:{\xypolygon6"A"{~<{}~>{-}{}}}
\POS"A4" \ar|-(0.8){\SelectTips{cm}{}\object@{+}}@{-}    "A0"
\endxy
\ar[ruu]\ar@{.}[rr]
&&\xy/r1pc/:{\xypolygon6"A"{~<{}~>{-}{}}}
\POS"A5" \ar@{-}   "A0"
\endxy
\ar[ruu]\ar@{.}[rr]
&&\xy/r1pc/:{\xypolygon6"A"{~<{}~>{-}{}}}
\POS"A6" \ar|-(0.8){\SelectTips{cm}{}\object@{+}}@{-}    "A0"
\endxy &
}
\]
\caption{The quiver $\Gamma_{\odot}$ for $p=2$ and $n=3$.}\label{gamma_np}
\end{figure}

\end{ejem}

\section{Cluster tilting theory for ${\mathcal C}_{n,p}$}

In this final section we are interested in understanding the cluster tilting objects of
$\mathcal{C}_{n,p}$, and compare them with configurations of tagged edges in the pounctured polygon
$\textsf{P}_{np}$. When $p=1$, it is known that cluster tilting objects of $\mathcal{C}_{n}$
correspond to  {\em triangulations} of a regular  polygon with $n$ vertices and one puncture; i.e. a maximal
collection of pairwise non crossing tagged edges.

Cluster tilting objects in $\mathcal{C}_{n,p}$ have  been studied from an
algebraic point of view in \cite{Zhu}.

In the following definition let $\mathrm{add}(T)$ be the full subcategory consisting
of direct summands of direct sums of finitely many copies of $T$.

\begin{defi}\cite{Zhu}
\label{deftilting} An object $T \in \mathcal{C}_{n,p}$ is called a {\em cluster
tilting object} if for any object $X\in\mathcal{C}_{n,p}$  we have that
\begin{enumerate}
  \item $\mathrm{Ext}_{\mathcal{C}_{n,p}}^1(T,X)=0$ if and only if $X\in\mathrm{add}(T)$;
  \item $\mathrm{Ext}_{\mathcal{C}_{n,p}}^1(X,T)=0$ if and only if $X\in\mathrm{add}(T)$.
\end{enumerate}

\end{defi}

A cluster tilting object is called basic if all its direct summands are
pairwise non isomorphic. In this paper all the cluster tilting objects we consider
will be basic, so we omit the term basic.

%Let denote by $\overline{\phi}: \mathcal{C}_{n,p}\rightarrow \mathcal{C}(\textsf{P}_{np})$  the isomorphism induced by  $\phi: \Gamma_{n,p}\rightarrow \Gamma_{\odot}$.

If $T=T_1\oplus T_2 \oplus \cdots\oplus T_k$ is an object in $ \mathcal{C}_{n,p}$ denote by $\mathcal{X}_T$ the set of tagged edges on the punctured polygon $\textsf{P}_{np}$ via the isomorphism  $\phi_p$

\[ T=T_1\oplus T_2 \oplus \cdots\oplus T_k \stackrel{\phi_p}{\longmapsto} \mathcal{X}_T=\{\phi_p(T_1),\cdots, \phi_p(T_k)\} \]

\vspace*{.5cm}

 Taking $p=1$ we have the next result that  follows from \cite{S}.

 \begin{lema}\label{tiltinC}
 $T$ is a cluster tilting object
in $\mathcal{C}_n$ if and only if $\mathcal{X}_T$ is a triangulation of  the regular punctured
polygon $\textsf{P}_{n}$. The cardinality of $\mathcal{X}_T$ is $n$.
\end{lema}

Now we are going to state a similar result for cluster tilting objects in $\mathcal{C}_{n,p}$.

If $\mathcal{X}_{T'}$ is a set of tagged edges $M_{a,b}^\ze$,  with $a,b\in \{1,\cdots, n\}$,   of $\textsf{P}_{n}$; by abuse of notation we  also denote by $\mathcal{X}_{T'}$ the set of tagged edges $M_{a,b}^\ze$ with $a,b\in \{1,\cdots, n\}$ of $\textsf{P}_{np}$. Recall that $\rho$ is the counterclockwise  rotation  through $\frac{2\pi}{p}$ around the center of $\textsf{P}_{np}$ ( which corresponds with the action of $F$ on $\mathcal{C}_{n,p}$). Then we have:

\begin{prop}\label{triangulacion}
$T$ is a cluster tilting object of $\mathcal{C}_{n,p}$ if and only if  there is a cluster tilting object $T'$ of $\mathcal{C}_{n}$ such that

$$\mathcal{X}_T=\mathcal{X}_{T'}\cup \rho(\mathcal{X}_{T'})\cup \cdots \cup \rho^{p-1}(\mathcal{X}_{T'})$$
\end{prop}

\begin{proof}
  Suppose that $T$ is a cluster tilting object in $\mathcal{C}_{n,p}$. Then by \cite[Theorem 3.5]{Zhu} $T'=\eta_p(T)$ is a cluster tilting object in $\mathcal{C}_{n}$. By Lemma \ref{tiltinC}, $\mathcal{X}_{T'}$ is a triangulation of  the regular punctured polygon $\textsf{P}_{n}$ with $n$ elements.
  Let $\mathcal{X}:=\mu_p^{-1}(\mathcal{X}_{T'})$  the corresponding  set of tagged edges  in $\textsf{P}_{n,p}$.
   Then $\mathcal{X}_{T}:=\phi_p(T)=\mu_p^{-1}\phi_1\eta_p(T)=\mathcal{X}$ and
   $\mathcal{X}=\mu_p^{-1}(\mathcal{X}_{T'})=\mathcal{X}_{T'}\cup \rho(\mathcal{X}_{T'})\cup \cdots \cup \rho^{p-1}(\mathcal{X}_{T'})$ by definition of $\mu_p$.

   On the other hand we assume that $\mathcal{X}_T=\mathcal{X}_{T'}\cup \rho(\mathcal{X}_{T'})\cup \cdots \cup \rho^{p-1}(\mathcal{X}_{T'})$ with $T'$ a cluster tilting object in $\mathcal{C}_n$. Then $T=\phi_p^{-1}(\mathcal{X}_T)=T'\oplus F(T')\oplus \cdots\oplus F^{p-1}(T')$ is a cluster tilting object in $\mathcal{C}_{n,p}$,  again by \cite[Theorem 3.5]{Zhu}.

   \end{proof}

\begin{defi}
  A set of tagged edges $\mathcal{X}$ of $\textsf{P}_{n,p}$ is said to be a \emph{$p$-triangulation} if  there is a triangulation $\mathcal{Y}$ of  $\textsf{P}_{n}$ (in the  sense of \cite{S}) such that $\mathcal{X}=\mu_p^{-1}(\mathcal{Y})$.
\end{defi}

%Observe that $\mu_p(\mathcal{X})= \mu_p(\rho(\mathcal{X}))=\cdots= \mu_p( \rho^{p-1}(\mathcal{X}))$ then if

   Then we can rewrite Proposition \ref{triangulacion} as follows:

  \begin{prop}
$T$ is a cluster tilting object of $\mathcal{C}_{n,p}$ if and only if  $\mathcal{X}_T$ is a $p$-triangulation of $\textsf{P}_{n,p}$.
  \end{prop}

  Since the $p$-triangulations have $pn$ different tagged edges we have the following Corollary.

  \begin{coro}
    Any cluster tilting object of $\mathcal{C}_{n,p}$ has $pn$ pairwise non isomorphic summands.
  \end{coro}
  \qed

\end{document}